\documentclass{amsart}
\usepackage[english ]{babel}
\usepackage{amsmath,latexsym,amssymb,verbatim}
\usepackage{pdfsync}

\usepackage[all]{xy}

\usepackage{hyperref}

\newtheorem{theorem}{Theorem}[section]
\newtheorem{lemma}[theorem]{Lemma}
\newtheorem{proposition}[theorem]{Proposition}
\newtheorem{corollary}[theorem]{Corollary}
\newtheorem{remark}[theorem]{Remark}
\newtheorem{remarks}[theorem]{Remarks}

\newtheorem{example}[theorem]{Example}

\newtheorem{definition}[theorem]{Definition}
\newtheorem{notation}[theorem]{Notation}

\DeclareMathOperator{\limi}{{lim}}
\newcommand{\ilim}[1]{\,\underset{#1}{\underset{\to}{\limi}}\,}
\newcommand{\plim}[1]{\,\underset{#1}{\underset{\leftarrow}{\limi}}\,}
\newcommand{\leftmapsto}{\leftarrow\!\shortmid}

\DeclareMathOperator{\Hom}{{Hom}}
\DeclareMathOperator{\Spec}{{Spec}}
\DeclareMathOperator{\Coker}{{Coker}}
\DeclareMathOperator{\Ker}{{Ker}}
\DeclareMathOperator{\Ima}{{Im}}

\NoCompileMatrices

\input arrow.tex

\begin{document}

\title{Flat SML  modules and reflexive functors}

\author{Carlos Sancho, Fernando Sancho, Pedro Sancho}

%

\date{May 9, 2017}

\begin{abstract}
We give some functorial characterizations of flat strict Mittag-Leffler  modules. We characterize reflexive functors of modules with similar tools,  definitions and theorems.

\end{abstract}

\maketitle

\section{Introduction}

Let $R$ be a commutative (associative with unit) ring. 
Let $\mathcal R$ be the covariant functor from the category of commutative $R$-algebras to the rings
defined by  $\mathcal R(S):=S$ for any commutative  $R$-algebra $S$.
Let $M$ be an  $R$-module. Consider the functor of $\mathcal R$-modules, $\mathcal M$, defined by $\mathcal M(S):=M\otimes_R S$, for any commutative $R$-algebra $S$. $\mathcal M$ is said to be the   {\sl quasi-coherent}  $\mathcal R$-module associated with $M$.The functors
$$\aligned \text{Category of $R$-modules } & \to \text{ Category of quasi-coherent $\mathcal R$-modules }\\ M & \mapsto \mathcal M\\ \mathcal M(R) & \leftarrow\!\shortmid \mathcal M\endaligned$$
stablish an equivalence  of categories.
 Consider the dual functor
$\mathcal M^*\!\!:=\!\mathbb Hom_{\mathcal R}(\mathcal M,\!\mathcal R)$ defined by
$\mathcal M^*(S):=\Hom_S(M\otimes_RS,S)$.   In general, the canonical morphism  $M\to M^{**}$ is not an isomorphism, but surprisingly  $\mathcal M=\mathcal M^{**}$ (see \ref{reflex}), that is, $\mathcal M$ is a reflexive functor of $\mathcal R$-module. 
This result has many applications in Algebraic Geometry
 (see \cite{Pedro}), for example the Cartier duality of commutative affine groups and commutative formal groups.

Given an $R$-module $N$ we shall say that $\mathcal N^*$ is an $\mathcal R$-module scheme. 
 In  \cite{Amel2}, we proved that an
 $R$-module $M$ is a finitely generated  projective module  iff  $\mathcal M$ is an $\mathcal R$-module scheme. In  \cite{Pedro2}, we proved that  $M$ is a flat $R$-module iff $\mathcal M$ is a direct limit of $\mathcal R$-module schemes. We proved too that the following statements are equivalent:
 
\begin{enumerate}
\item  $M$ is a flat Mittag-Leffler module

\item $\mathcal M$ is the direct limit of its $\mathcal R$-submodule schemes.

\item The kernel of any morphism $\mathcal N^*\to \mathcal M$ 
is an $\mathcal R$-module scheme.
 
\item The kernel of any morphism $\mathcal R^n\to \mathcal M$ 
is an $\mathcal R$-module  scheme.

\end{enumerate}

In this paper we shall give some functorial characterizations  of flat strict Mittag-Leffler modules.
Mittag-Leffler conditions were first introduced by Grothendieck in \cite{Grot}, and deeply studied
by some authors, for example, Raynaud and Gruson in \cite{RaynaudG}. Flat strict Mittag-Leffler modules have also been
studied by Ohm and Rush under the name of ''trace modules'' in \cite{Ohm}, by Garfinkel, who calls
them ``universally torsionless'' in \cite{Garfinkel} and by Zimmermann-Huisgen, under the name of ``locally
projective modules'' in \cite{Zimmermann-Huisgen}. We prove the following theorem.
 
\begin{theorem} Let $M$ be an $R$-module. The following statements are equivalent.

 \begin{enumerate}

\item $M$ is a flat strict Mittag-Leffler module (see \cite[II 2.3.2]{RaynaudG}). That is, $M$ is flat and it is  isomorphic to a direct limit of finitely presented modules $F_i$, so that
for every $R$-module $N$ and every $i$ there exists a $j\geq i$ such that $$\Ima(\Hom_R(M,N)\to \Hom(F_i,N))=
\Ima(\Hom_R(F_j,N)\to \Hom(F_i,N)).$$

\item $\mathcal M=\ilim{i} \mathcal N_i^*$, where $\{\mathcal N_i^*\}$ is the set of the $\mathcal R$-submodule schemes of $\mathcal M$, and the natural morphisms $M^*\to N_i$ are surjective.

\item $\mathcal M^*$ is  dually separated, that is, the natural morphism
$$M\otimes_RS\to (M\otimes_RS)^{**}:=\Hom_S(\Hom_S(M\otimes_R S,S),S)$$
is injective, for any commutative $R$-algebra $S$.

\item The natural morphism $M\otimes_R N\to \Hom_R(M^*,N)$ is injective for every $R$-module $N$ (that is, $M$ is universally torsionless, see \cite{Garfinkel}).

\item There exists a monomorphism $\mathcal M\to \prod^I \mathcal R$.

\item $M$ is a flat Mittag-Leffler module and the morphism
$$M\otimes_R R/\mathfrak m\to \Hom_R(M^*,R/\mathfrak m)$$
is injective,  for every maximal ideal $\mathfrak m\subset R$.

\item The cokernel of every morphism $\mathcal M^*\to \mathcal N$ is quasi-coherent, for every quasi-coherent $\mathcal R$-module $\mathcal N$.

\item The cokernel of every morphism $\mathcal M^*\to \mathcal R$ is quasi-coherent (which is equivalent to saying that $M$ is a trace module, see \ref{trace2}).

\end{enumerate}
\end{theorem}

More generally we shall give some characterizarions of dually separated functors of $\mathcal R$-modules. 

\begin{theorem} Let $\mathbb M$ be a functor of $\mathcal R$-modules. The following statements are equivalent

\begin{enumerate}

\item $\mathbb M$ is  dually separated: The natural morphism $\mathbb M^*(S)\to \Hom_S(\mathbb M(S),S)$ is injective, for any commutative $R$-algebra $S$.                

\item The natural morphism $\Hom_{\mathcal R} (\mathbb M,\mathcal N)\to  \Hom_{R } (\mathbb M(R),N)$ is injective, for any $R$-module  $N$.

\item The natural morphism $\Hom_{\mathcal R} (\mathbb M,\mathbb N)\to  \Hom_{R } (\mathbb M(R),\mathbb N(R))$ is injective, for any dual functor $\mathbb N$.

\item The cokernel of every morphism $\mathbb M\to  \mathcal N$ is quasi-coherent, for any $R$-module $N$.

\hskip -1.75cm Assume that $\mathbb M$ is reflexive.

\item There exists a monomorphism $\mathbb M^*\to \prod^I\mathcal R$.

\hskip -1.75cm Now assume that $R$ is a field.

\item $\mathbb M^*=\ilim{i} \mathcal N_i^*$, where $\{\mathcal N_i\}$ is the set of the quasi-coherent quotient $\mathcal R$-modules of $\mathbb M$.

\end{enumerate}

\end{theorem}

If $R$ is a field and $\mathbb M$ is a reflexive functor of $\mathcal R$-modules, we prove that $\mathbb M$ is  dually separated   and we obtain the following theorem.

\begin{theorem} Let $R=K$ be a field.  A functor of  $\mathcal K$-modules is reflexive iff
it is equal to the inverse limit of its quasi-coherent quotient $\mathcal R$-modules.
\end{theorem}

If $I$ is a totally ordered set and $\{f_{ij}\colon  M_i\to M_j\}_{i\geq j\in I}$ is an inverse system of  $K$-vector spaces, we prove that $\plim{i\in I} \mathcal M_i$ is a reflexive functor of $\mathcal K$-modules.
Unfortunately, we do not know if arbitrary inverse limits of quasi-coherent  $\mathcal K$-modules are reflexive.

\section{Preliminaries}

Let $R$ be a commutative ring (associative with a unit). All the  functors considered in this paper are covariant functors from the category of commutative $R$-algebras (always assumed to be associative with a unit) to the category of sets. A functor $\mathbb X$ is said to be a functor of sets (resp. groups, rings,  etc.) if $\mathbb X$ is a functor from the category of  commutative $R$-algebras to the category of sets (resp. groups, rings, etc.).

\begin{notation} \label{nota2.1}For simplicity, given a (covariant) functor $\mathbb X$ (from the category of commutative $R$-algebras to the category of sets),
we shall sometimes use $x \in \mathbb X$  to denote $x \in \mathbb X(S)$. Given $x \in \mathbb X(S)$ and a morphism of commutative $R$-algebras $S \to S'$, we shall still denote by $x$ its image by the morphism $\mathbb X(S) \to \mathbb X(S')$.\end{notation}

 Let $\mathbb M$ and $\mathbb M'$ be two ${\mathcal R}$-modules.
 A morphism of ${\mathcal R}$-modules $f\colon \mathbb M\to \mathbb M'$
 is a morphism of functors  such that the  morphism $f_{S}\colon \mathbb M({S})\to
 \mathbb M'({S})$ defined by $f$ is a morphism of ${S}$-modules, for any commutative $R$-algebra ${S}$.
 We shall denote by $\Hom_{{\mathcal R}}(\mathbb M,\mathbb M')$ the  family of all the morphisms of ${\mathcal R}$-modules from $\mathbb M$ to $\mathbb M'$.

\begin{remark} Direct limits, inverse limits of ${\mathcal R}$-modules and  kernels, cokernels, images, etc.,  of morphisms of ${\mathcal R}$-modules are regarded in the category of ${\mathcal R}$-modules.\end{remark}

One has
$$\aligned & 
(\Ker f)({S})=\Ker f_{S},\, (\Coker f)({S})=\Coker f_{S},\, (\Ima f)({S})=\Ima f_{S},\\
&  (\ilim{i\in I} \mathbb M_i)({S})=\ilim{i\in I} (\mathbb M_i({S})),\,
(\plim{j\in J} \mathbb M_j)({S})=\plim{j\in J} (\mathbb M_j({S})),\endaligned $$
(where $I$ is an upward directed set and $J$ a downward directed set).
$\mathbb M\otimes_{{\mathcal R}}\mathbb M'$ is defined by $(\mathbb M\otimes_{{\mathcal R}}\mathbb M')({S}):=\mathbb M({S})\otimes_{{S}}\mathbb M'({S})$, for any commutative $R$-algebra ${S}$.

\begin{definition}  Given an ${\mathcal R}$-module $\mathbb M$ and a commutative $R$-algebra ${S}$,   we shall denote by $\mathbb M_{|{S}}$  the restriction of $\mathbb M$ to the category of commutative ${S}$-algebras, i.e.,  $$\mathbb M_{\mid {S}}(S'):=\mathbb M(S'),$$ for any commutative ${S}$-algebra $S'$.
\end{definition}

We shall denote by ${\mathbb Hom}_{{\mathcal R}}(\mathbb M,\mathbb M')$\footnote{In this paper, we shall only  consider well-defined functors ${\mathbb Hom}_{{\mathcal R}}(\mathbb M,\mathbb M')$, that is to say, functors such that $\Hom_{\mathcal {S}}(\mathbb M_{|{S}},\mathbb {M'}_{|{S}})$ is a set, for any ${S}$.} the  ${\mathcal R}$-module defined by $${\mathbb Hom}_{{\mathcal R}}(\mathbb M,\mathbb M')({S}):={\rm Hom}_{\mathcal {S}}(\mathbb M_{|{S}}, \mathbb M'_{|{S}}).$$  Obviously,
$$(\mathbb Hom_{{\mathcal R}}(\mathbb M,\mathbb M'))_{|{S}}=
\mathbb Hom_{\mathcal {S}}(\mathbb M_{|{S}},\mathbb M'_{|{S}}).$$

\begin{notation} \label{nota2.2} Let $\mathbb M$ be an ${\mathcal R}$-module. We shall denote $\mathbb M^*=\mathbb Hom_{{\mathcal R}}(\mathbb M,\mathcal R)$.\end{notation}

\begin{proposition} \label{trivial} Let $\mathbb M$ and $\mathbb N$ be two  ${\mathcal R}$-modules. Then,
$$\Hom_{{\mathcal R}}(\mathbb M,\mathbb N^*)=\Hom_{{\mathcal R}}(\mathbb N,\mathbb M^*),\, f\mapsto \tilde f,$$
where $\tilde f$ is defined as follows: $\tilde f(n)(m):=f(m)(n)$, for any $m\in\mathbb M$ and $n\in\mathbb N$.

\end{proposition}

\begin{proof} $\Hom_{{\mathcal R}}(\mathbb M,\mathbb N^*)=\Hom_{{\mathcal R}}(\mathbb M\otimes_{{\mathcal R}}\mathbb N,\mathcal R)=\Hom_{{\mathcal R}}(\mathbb N,\mathbb M^*)$.

\end{proof}

\begin{proposition} \cite[1.15]{Amel}  \label{adj2} Let $\mathbb M$ be an $\mathcal R$-module, $S$ a commutative $R$-algebra and $N$ an $S$-module.  Then,
$$\Hom_{\mathcal R}(\mathbb M,\mathcal N)=\Hom_{\mathcal S}(\mathbb M_{|S},\mathcal N), w\mapsto \pi\circ w_{|S},$$ where $\pi\colon \mathcal N_{|S}\to \mathcal N$ is defined by $\pi_T(n\otimes_R t):=n\otimes_S t\in N\otimes_ST$, for any commutative $S$-algebra $T$ and any $n\otimes_Rt\in N\otimes_RT=\mathcal N_{|S}(T)$.
In particular,
$$\Hom_{\mathcal R}(\mathbb M,\mathcal S)=\mathbb M^*(S).$$
\end{proposition}

\subsection{Quasi-coherent modules}

\begin{definition} Let $M$ (resp. $N$,  $V$, etc.) be an $R$-module. We shall denote by  ${\mathcal M}$  (resp. $\mathcal N$, $\mathcal V$, etc.)   the ${\mathcal R}$-module defined by ${\mathcal M}({S}) := M \otimes_R {S}$ (resp. $\mathcal N({S}):=N\otimes_R {S}$, etc.). $\mathcal M$  will be called the quasi-coherent ${\mathcal R}$-module associated with $M$. 
\end{definition}

 ${\mathcal M}_{\mid {S}}$ is the quasi-coherent $\mathcal {S}$-module associated with $M \otimes_R
{S}$.  For any pair of $R$-modules $M$ and $N$, the quasi-coherent module associated with $M\otimes_R N$ is $\mathcal M\otimes_{{\mathcal R}}\mathcal N$.

\begin{proposition} \cite[1.12]{Amel}  The functors
$$\aligned \text{Category of $R$-modules } & \to \text{ Category of quasi-coherent $\mathcal R$-modules }\\ M & \mapsto \mathcal M\\ \mathcal M(R) & \leftmapsto \mathcal M\endaligned$$
stablish an equivalence  of categories. In particular,
$${\rm Hom}_{{\mathcal R}} ({\mathcal M},{\mathcal M'}) = {\rm Hom}_R (M,M').$$
\end{proposition}

Let $f\colon M\to N$ be a morphism of $R$-modules and $\tilde f\colon \mathcal M\to \mathcal N$
the associated morphism of $\mathcal R$-modules. Let $C=\Coker f$, then $\Coker\tilde f=\mathcal C$, which is a quasi-coherent module.

\begin{proposition} \cite[1.3]{Amel}\label{tercer}
For every  ${{\mathcal R}}$-module $\mathbb M$ and every $R$-module $M$, it is satisfied that
$${\rm Hom}_{{\mathcal R}} ({\mathcal M}, \mathbb M) = {\rm Hom}_R (M, \mathbb M(R)),\, f\mapsto f_R.$$
\end{proposition}

\begin{notation} Let $\mathbb M$ be an ${\mathcal R}$-module. 
We shall denote by $\mathbb M_{qc}$ the quasi-coherent module associated with
the $R$-module $\mathbb M(R)$, that is, $$\mathbb M_{qc}({S}):=\mathbb M(R)\otimes_R{S}.$$\end{notation}

\begin{proposition} \label{tercerb} For each  ${\mathcal R}$-module $\mathbb M$ one has the natural morphism $$\mathbb M_{qc}\to \mathbb M, \,m\otimes s\mapsto s\cdot m,$$ for any $m\otimes s\in \mathbb M_{qc}({S})=\mathbb M(R)\otimes_R {S}$, and a functorial equality 
$$\Hom_{{\mathcal R}}(\mathcal N,\mathbb M_{qc}))=\Hom_{{\mathcal R}}(\mathcal N,\mathbb M,$$
for any quasi-coherent ${\mathcal R}$-module $\mathcal N$.
\end{proposition}

\begin{proof} Observe that $\Hom_{{\mathcal R}}(\mathcal N,\mathbb M)\overset{\text{\ref{tercer}}}=\Hom_R(N,\mathbb M(R))\overset{\text{\ref{tercer}}}=\Hom_{{\mathcal R}}(\mathcal N,\mathbb M_{qc})$.

\end{proof}

Obviously, an $\mathcal R$-module $\mathbb M$ is a quasi-coherent module iff
the natural morphism $\mathbb M_{qc}\to \mathbb M$ is an isomorphism.

\begin{theorem}  \cite[1.8]{Amel}\label{prop4} 
Let $M$ and $M'$ be $R$-modules. Then, $${\mathcal M} \otimes_{{\mathcal R}} {\mathcal M'}={\mathbb Hom}_{{\mathcal R}} ({\mathcal M^*}, {\mathcal M'}),\, m\otimes m'\mapsto \tilde{m\otimes m'},$$
where $ \tilde{m\otimes m'}(w):=w(m)\cdot m'$, for any $w\in \mathcal M^*$.
\end{theorem}

If we make $\mathcal M'={\mathcal R}$ in the previous theorem, we obtain the following theorem.

\begin{theorem} \cite[II,\textsection 1,2.5]{gabriel}  \cite[1.10]{Amel}\label{reflex}
Let $M$ be an $R$-module. Then $${\mathcal M}={\mathcal M^{**}}.$$
\end{theorem}

\begin{definition} Let $\mathbb M$ be  an ${\mathcal R}$-module. We shall say that
$\mathbb M^*$ is a dual functor.
We shall say that  an ${\mathcal R}$-module  ${\mathbb M}$ is reflexive if ${\mathbb M}={\mathbb M}^{**}$.\end{definition}

\begin{example}  Quasi-coherent modules are reflexive.\end{example}

\subsection{$\mathcal R$-module schemes} 

\begin{definition} Let $M$ be an $R$-module. $\mathcal M^*$  will be called the ${\mathcal R}$-module scheme associated with $M$.
\end{definition}

\begin{definition} Let $\mathbb N$ be an ${\mathcal R}$-module. 
We shall denote by $\mathbb N_{sch}$ the ${\mathcal R}$-module scheme defined by $$\mathbb N_{sch}:=((\mathbb N^*)_{qc})^*.$$\end{definition}

\begin{proposition} \label{1211b} Let $\mathbb N$ be an ${\mathcal R}$-module. Then,
\begin{enumerate}
\item  $\mathbb N_{sch}({S})=\Hom_R(\mathbb N^*(R),{S})$.

\item $
\Hom_{{\mathcal R}}(\mathbb N_{sch},\mathcal M)=\mathbb N^*(R)\otimes_R M$,
for any quasi-coherent module $\mathcal M$.

\end{enumerate}

\end{proposition}

\begin{proof} 1. $\mathbb N_{sch}({S})=\Hom_{{\mathcal R}}((\mathbb N^*)_{qc},\mathcal {S})=\Hom_R(\mathbb N^*(R),{S})$.

2. $\Hom_{{\mathcal R}}(\mathbb N_{sch},\mathcal M)\overset{\text{\ref{prop4} }}=(\mathbb N^*)_{qc}(R)\otimes_RM=\mathbb N^*(R)\otimes_R M$.

\end{proof}

The natural morphism $(\mathbb N^*)_{qc} \to \mathbb N^*$ corresponds by Proposition \ref{trivial} with a morphism $$\mathbb N\to \mathbb N_{sch}.$$ Specifically, one has the natural morphism
$$\begin{array}{lll} \mathbb N({S}) & \to & \Hom_{R}(\mathbb N^*(R),{S})=\mathbb N_{sch}({S})\\ n & \mapsto & \tilde n,\text{ where } \tilde n(w):=w_{S}(n)\end{array}$$

\begin{proposition} \label{1211} Let $\mathbb N$ be  an ${\mathcal R}$-module and $M$ an $R$-module. Then, the natural morphism
 $$\Hom_{{\mathcal R}}(\mathbb N,\mathcal M^*)\to
\Hom_{{\mathcal R}}(\mathbb N_{sch},\mathcal M^*),$$
is an isomorphism.

\end{proposition}

\begin{proof} $\Hom _{{\mathcal R}}(\mathbb N, \!\mathcal M^*) \! \overset{\text{\ref{trivial}}} =\!
\Hom_{{\mathcal R}}(\mathcal M,\!\mathbb N^*)\! \overset{\text{\ref{tercerb}}}=\!\Hom_{{\mathcal R}}(\mathcal M,\! (\mathbb N^*)_{qc})\!{\overset{\text{\ref{trivial}}}=
\Hom_{{\mathcal R}}(\mathbb N_{sch},\mathcal M^*)}.$

\end{proof}

\section{Dually separated   $\mathcal R$-modules}

\begin{definition} We shall say that \label{dualq} an $\mathcal R$-module $\mathbb M$ is  dually separated    if  the natural morphism $\mathbb M^*\to {\mathbb M_{qc}}^*$ is  a monomorphism.\end{definition}

\begin{example} Quasi-coherent modules, $\mathcal M$, are  dually separated, because $\mathcal M^{*}={\mathcal M_{qc}}^*$.
\end{example}

\begin{example} \label{3.3.3} If $M=\oplus_I R$ is a free $R$-module, then $\mathcal M^*$ is  dually separated   : 
The obvious monomorphism $\mathcal M=\oplus_I \mathcal R\to \prod_I\mathcal R$, factors through  $\mathcal M\to \mathcal M_{sch}$, by Proposition \ref{1211}. Hence, the morphism $\mathcal M\to \mathcal M_{sch}$ is a monomorphism. That is, $\mathcal M^{**}=\mathcal M
\to  {{\mathcal M^*}_{qc}}^*$ is  a monomorphism and $\mathcal M^*$ is  dually separated.
\end{example}

\begin{proposition}  \label{3.3.4} The direct limit of a direct system of  dually separated   $\mathcal R$-modules is  dually separated. Every quotient of  a dually separated  $\mathcal R$-module   is  dually separated.\end{proposition}

\begin{proof} Let $\mathbb M=\ilim{i} \mathbb M_i$ be a direct limit of  dually separated   $\mathcal R$-modules. Then, the morphism
$$\mathbb M^*=\plim{i}\mathbb M_i^*\hookrightarrow \plim{i} {\mathbb M_{i,qc}}^*=(\ilim{i} \mathbb M_{i,qc})^*={\mathbb M_{qc}}^*$$
is  a monomorphism. Then, $\mathbb M$ is  dually separated.

Let $\mathbb M$ be  dually separated    and $\mathbb M\to \mathbb N$ an epimorphism. 
The morphism $\mathbb N^* \to {\mathbb N_{qc}}^*$ is  a monomorphism because the diagram 
$$\xymatrix{\mathbb N^* \ar@{^{(}->}[d] \ar[r] & {\mathbb N_{qc}}^* \ar@{^{(}->}[d]\\
\mathbb M^* \ar@{^{(}->}[r]& {\mathbb M_{qc}}^*}$$
is commutative. Then, $\mathbb N$ is  dually separated.
\end{proof}

\begin{proposition} \label{QU} If $\mathbb M$ is  a dually separated    $\mathcal R$-module and $S$ is a commutative $R$-algebra, then the  $\mathcal S$-module $\mathbb M_{|S}$ is  dually separated.\end{proposition}

\begin{proof} Let $S$ be a commutative $R$-algebra and let $T$ be a commutative $T$-algebra. The diagram
$$\xymatrix{{\mathbb M_{|S}}^*(T)=\Hom_{\mathcal T}(\mathbb M_{|T},\mathcal T) \ar@{=}[r] \ar[d] &
\mathbb M^*(T)\ar@{^{(}->}[d]
\\ {\mathbb M_{|S, qc}}^*(T)= \Hom_{\mathcal S}(\mathbb M(S),T) \ar[r] & \Hom_{R}(\mathbb M(R),T)={\mathbb M_{qc}}^*(T)}$$
is commutative, then the morphism ${\mathbb M_{|S}}^*\to  {\mathbb M_{|S,qc}}^*$ is  a monomorphism. \end{proof}

\begin{theorem}  \label{W1} An $\mathcal R$-module $\mathbb M$ is  dually separated   
 iff the map
$$\Hom_{\mathcal R}(\mathbb M,\mathcal N)\to \Hom_{R}(\mathbb M(R),N),\quad f\mapsto f_R$$
is  injective,  for any $R$-module $N$.\end{theorem}

\begin{proof} If the natural morphism $\mathbb M^*\to {\mathbb M_{qc}}^*$ is  a monomorphism, then
$$\Hom_{\mathcal R}(\mathbb M,\mathcal S)\hookrightarrow \Hom_{R}(\mathbb M(R),S),$$ is injective for any commutative $R$-algebra $S$. Given an $R$-module $N$, consider the $R$-algebra $S:=R\oplus N$, with the multiplication operation $(r,n)\cdot (r',n'):=(rr',rn'+r'n)$. The composite morphism
$$\Hom_{\mathcal R}(\mathbb M, \mathcal R\oplus\mathcal N)=\Hom_{\mathcal R}(\mathbb M,\mathcal S)\hookrightarrow \Hom_{R}(\mathbb M(R),S)= \Hom_{R}(\mathbb M(R), R\oplus N)$$ is injective. 
Hence, $\Hom_{\mathcal R}(\mathbb M,\mathcal N)\to \Hom_{R}(\mathbb M(R),  N)$ is injective.

Reciprocally,
$\mathbb M^*(S)=\Hom_{\mathcal R}(\mathbb M,\mathcal S)\hookrightarrow
\Hom_{R}(\mathbb M(R),S)={\mathbb M_{qc}}^*(S)$ is injective
for any commutative $R$-algebra $S$, hence the morphism $\mathbb M^*\to {\mathbb M_{qc}}^*$ is  a monomorphism.
\end{proof}

\begin{theorem} \label{QHom} Let $\mathbb M$ be an $\mathcal R$-module. $\mathbb M$ is  dually separated    iff the morphism
$$\Hom_{\mathcal R}(\mathbb M,\mathbb M')\to \Hom_{R}(\mathbb M(R),\mathbb M'(R)),\quad f\mapsto f_R$$
is injective, for every dual  $\mathcal R$-module $\mathbb M'=\mathbb N^*$.\end{theorem}

\begin{proof} $\Rightarrow)$ From the commutative diagram
$$\xymatrix{\Hom_{\mathcal R}(\mathbb M,\mathbb M') \ar@{=}[r]^-{\text{\ref{trivial}}} \ar[d]
& \Hom_{\mathcal R}(\mathbb N,\mathbb M^*) \ar@{^{(}->}[r]^-{\text{\ref{dualq}}}
&   \Hom_{\mathcal R}(\mathbb N,{\mathbb M_{qc}}^*)\\
\Hom_{R}(\mathbb M(R),\mathbb M'(R)) \ar@{=}[r]^-{\text{\ref{tercer}}} & \Hom_{\mathcal R}({\mathbb M_{qc}},\mathbb M')
\ar@{=}[ur]^-{\text{\ref{trivial}}}&}$$
 one deduces that the morphism $\Hom_{\mathcal R}(\mathbb M,\mathbb M')\to \Hom_{R}(\mathbb M(R),\mathbb M'(R))$ is injective.
 
 $\Leftarrow)$ It is an immediate consequence of Theorem \ref{W1}.\end{proof}

\begin{proposition} \label{P10}
Let $\mathbb A$ be an ${\mathcal R}$-algebra and  dually separated, let $\mathcal M$ and $\mathcal N$ be $\mathbb A$-modules and let $M'$ be a direct summand of $M$.  Then,
\begin{enumerate}
\item $\mathcal M'$ is a quasi-coherent
$\mathbb A$-submodule of $\mathcal M$ iff $M'$ is an
$\mathbb A(R)$-submodule of $M$.

\item A morphism $f\colon \mathcal M\to \mathcal N$ of $\mathcal R$-modules is a morphism of $\mathbb A$-modules iff $f_R\colon M\to N$ is a morphism of $\mathbb A(R)$-modules.

\end{enumerate}

\end{proposition}

\begin{proof} $(1)$
Obviously, if $\mathcal M'$ is an $\mathbb A$-submodule of $\mathcal M$ then $M'$ is
an $\mathbb A(R)$-submodule of $M$. Inversely, assume  $M=M'\oplus M''$ and   assume $M'$ is an
$\mathbb A(R)$-submodule of $M$. Let us consider the morphism 
$h\colon \mathbb A \to \mathbb Hom_{\mathcal R}( {\mathcal M}', {\mathcal M})$, $h(a):=a\cdot $. Write 
$$ \mathbb Hom_{\mathcal R}( {\mathcal M}', {\mathcal M})=  \mathbb Hom_{\mathcal R}( {\mathcal M}', {\mathcal M'})\times  \mathbb Hom_{\mathcal R}( {\mathcal M}', {\mathcal M''})$$
and write $h=(h_1,h_2)$. As $h_R=({h_1}_R,0)$, then $h_2=0$ and $\mathcal M'$ is an 
$\mathbb A$-submodule of $\mathcal M$. 

$(2)$ The morphism $f$ is a morphism of $\mathbb A$-modules iff $F\colon \mathbb A\otimes \mathcal M\to \mathcal N$, $F(a\otimes m):=f(am)-af(m)$ is the zero morphism. Likewise, $f_R$ is a morphism of $\mathbb A(R)$-modules iff $F_R\colon \mathbb A(R)\otimes M\to N$, $F_R(a\otimes m)=f_R(am)-af_R(m)$ is the zero morphism. Now, it easy to conclude the proof because the composite morphism 
$$\aligned \Hom_{\mathcal R}(\mathbb A\otimes \mathcal M,\mathcal N) & =
\Hom_{\mathcal R}(\mathbb A,\mathbb Hom_{\mathcal R}(\mathcal M,\mathcal N))
\underset{\text{\ref{QHom}}}\hookrightarrow \Hom_{R}(\mathbb A(R),\Hom_{\mathcal R}(\mathcal M,\mathcal N))
\\ & =  \Hom_{R}(\mathbb A(R),\Hom_{R}(M,N))=
  \Hom_{R}(\mathbb A(R)\otimes M,N)\endaligned$$
is injective.

\end{proof}

\begin{example} Let $G=\Spec A$ be an affine group $R$-scheme. The category of comodules over $A$ is equivalent to the category of quasi-coherent $G^\cdot$-modules ($G^\cdot$ is the functor defined by $G^\cdot(S)=\Hom_{R-alg}(A,S)$).  The category of quasi-coherent $G^\cdot$-modules is equal to the category of quasi-coherent $\mathcal A^*$-modules (see \cite[5.5]{Amel}). Let $M$ and $N$ be $A$-comodules and $f\colon M\to N$ a  morphism of $R$-modules. Then, $f$ is a morphism of $A$-comodules iff $f$ is a morphism of $A^*$-modules.
A direct summand  $M'\subseteq M$ is a $A$-subcomodule iff $ M'$ is an $ A^*$-submodule of $M$.

\end{example}

\begin{proposition} If $\mathbb M$ and $\mathbb M'$ are  dually separated,
$\mathbb M\otimes_{\mathcal R} \mathbb M'$ is  dually separated.

\end{proposition}

\begin{proof} Let $\mathbb N$ be a dual $\mathcal R$-module. Then,
the composite morphism
$$\aligned & \Hom_{\mathcal R}(\mathbb M\otimes_{\mathcal R}\mathbb M',\mathbb N)=
\Hom_{\mathcal R}(\mathbb M,\mathbb Hom_{\mathcal R}(\mathbb M',\mathbb N))\\ & 
\overset{\text{\ref{QHom}}}\hookrightarrow \Hom_{\mathcal R}(\mathbb M(R),\Hom_{\mathcal R}(\mathbb M',\mathbb N))
\overset{\text{\ref{QHom}}}\hookrightarrow \Hom_{\mathcal R}(\mathbb M(R),\Hom_{\mathcal R}(\mathbb M'(R),\mathbb N(R))) \\ & =\Hom_{\mathcal R}(\mathbb M(S)\otimes_{R}\mathbb M'(S),\mathbb N(S)).
\endaligned$$
is injective. Hence, $\mathbb M\otimes_{\mathcal R} \mathbb M'$ is  dually separated, by Theorem \ref{QHom}.

\end{proof}

\begin{lemma} \label{lemaD} An  $\mathcal R$-module $\mathbb M$ is  dually separated     iff the cokernel of every $\mathcal R$-module morphism from $\mathbb M$ to a quasi-coherent module is quasi-coherent, that is, the cokernel of any morphism
$f\colon \mathbb M\to \mathcal N$ is the quasi-coherent module associated with $\Coker f_R$.\end{lemma}

\begin{proof} $\Rightarrow)$ Let $f\colon \mathbb M\to \mathcal N$ be a morphism of $\mathcal R$-modules. Let $N':=\Coker f_R$. $\Coker f$ is a quotient $\mathcal R$-module  of $\mathcal N'$.
 Let $\pi\colon \mathcal N\to \mathcal N'$ be the natural epimorphism.  As $(\pi\circ f)_R=0$,   $\pi\circ f=0$ by Theorem \ref{W1}. Then, $\Coker f=\mathcal N'$.

$\Leftarrow)$ Let $f\colon \mathbb M\to \mathcal N$ be a morphism of $\mathcal R$-modules. If $f_R=0$ then $\Coker f=\mathcal N$ and $f=0$. Therefore,
$\mathbb M$ is  dually separated, by Theorem \ref{W1}.

\end{proof}

\begin{theorem} \label{4.3B} Let $\mathbb M$ be an $\mathcal R$-module. 
$\mathbb M$ is  dually separated   iff the natural morphism
$$\mathbb M^*(S)\to \Hom_{S}(\mathbb M(S), S), $$
is injective, for any commutative $R$-algebra $S$.
\end{theorem}

\begin{proof} $\Rightarrow)$ $\mathbb M^*(R)\to \Hom_{R}(\mathbb M( \mathcal R), R) $ is injective because $\mathbb M$ is dually separated.  $\mathbb M_{|S}$  is dually separated, by Proposition \ref{QU}. Then, the morphism
$$\mathbb M^*(S)={\mathbb M_{|S}}^*(S)\to \Hom_{S}({\mathbb M_{|S}}(S), S)= \Hom_{S}(\mathbb M(S), S), $$
is injective.

$\Leftarrow)$ 
Let $N$ be an $R$-module. Consider the commutative $R$-algebra $S=R\oplus N$ ($(r,n)\cdot (r',n'):=(rr',rn'+r'n)$), the morphism $\pi_1\colon  S\to R$, $\pi_1(r,n)=r$, the obvious morphism
$\pi_{1,*}\colon \mathbb M(S)\to \mathbb M(R)$ and the induced morphism $$\pi_{1,N}^*\colon \Hom_R(\mathbb M(R), N)\to
\Hom_S(\mathbb M(S), N), \, \pi_{1,N}^*(v)=v\circ \pi_{1,*}.$$ Let $\pi\colon \mathcal N_{|S}\to \mathcal N$ be defined by $\pi_T(n\otimes_R t):=n\otimes_S t$, for any commutative $S$-algebra $T$ and
$n\otimes_R t\in N\otimes_R T$. The diagram 
$$\xymatrix{\Hom_{\mathcal R}(\mathbb M,\mathcal N) \ar@{=}[d]^-{\text{\ref{adj2}}} \ar[r] & \Hom_R(\mathbb M(R), N) \ar[d]_-{\pi_{1,N}^*} \\ \Hom_{\mathcal S}({\mathbb M}_{|S},\mathcal N) \ar[r] & \Hom_S(\mathbb M(S), N)} \xymatrix{ w \ar@{|->}[r]  \ar@{|->}[d] & w_R \ar@{|-->}[d]_-{\pi_{1,N}^*}\\ \pi\circ w_{|S} \ar@{|->}[r] & \pi_S\circ w_S}$$
is commutative, because the diagram
$$\xymatrix{ \mathbb M(R) \ar[rr]^-{w_R} \ar[d] && N \ar[d] \ar@{=}[rd]  & \\
\mathbb M(S) \ar[rr]^-{w_S} \ar[d]^-{\pi_{1,*}} \ar@{-->}[rrd]^-{\pi_{1,N}^*(w_R)} & & N\otimes_RS \ar[d]\ar[r]^-{\pi_S} & N\otimes_SS=N\\
 \mathbb M(R) \ar[rr]^-{w_R}  & &  N  \ar@{=}[ru] &  }$$
is commutative, therefore $\pi_{1,N}^*(w_R)=w_R\circ \pi_{1,*}=\pi_S\circ w_S$. The diagram 
$$\xymatrix{ & \Hom_{\mathcal R}(\mathbb M,\mathcal N)  \ar[r] \ar@{=}[d]^-{\text{\ref{adj2}}} & \Hom_R(\mathbb M(R),N) \ar[d]_-{\pi_{1,N}^*}
\\ & \Hom_{\mathcal S}({\mathbb M}_{|S},\mathcal N) \ar@{^{(}->}[d] \ar[r] & \Hom_S(\mathbb M(S), N) \ar@{^{(}->}[d] \\
\mathbb M^{*}(S) \ar@{=}[r] &
\Hom_{\mathcal S}({\mathbb M}_{|S},\mathcal S)\ar@{^{(}->}[r]  & \Hom_{S}(\mathbb M(S), S)}$$
is commutative, then the morphism  $\Hom_{\mathcal R}(\mathbb M,\mathcal N)  \to \Hom_R(\mathbb M(R),N)$ is injective. By Theorem  \ref{5.1},
$\mathbb M$   is   dually separated.

\end{proof}

\begin{theorem} \label{3.14} Let $R=K$ be a field.  A $\mathcal K$-module, $\mathbb M$, is  dually separated    iff for every quasi-coherent $\mathcal K$-module $\mathcal N$, the image of every morphism $f\colon \mathbb M\to \mathcal N$ is a quasi-coherent $\mathcal K$-module.\end{theorem}

\begin{proof} The kernel of every morphism between quasi-coherent $\mathcal K$-modules is quasi-coherent. Then, the cokernel of a morphism $f\colon \mathbb M\to \mathcal N$ is quasi-coherent iff $\Ima f$ is quasi-coherent. This theorem is a consequence of Lemma \ref{lemaD}.

\end{proof}

%
%
%

\begin{lemma} \cite[1.28]{Pedro2} \label{L5.11} It holds that $$\Hom_{\mathcal R}(\mathcal N^*,\ilim{i} \mathcal M^*_i)=
\ilim{i}\Hom_{\mathcal R}(\mathcal N^*, \mathcal M^*_i).$$

\end{lemma}

\begin{theorem} \label{P2}  Let $R=K$ be a field.
Let $\mathbb M$ be a $\mathcal K$-module and let $\{\mathcal N_i\}_{i\in I}$ be the family of  all
the quasi-coherent quotient modules of $\mathbb M$. 
Then, $\mathbb M$ is  dually separated     iff $I$ is a downward directed set (in the obvious way) and $\mathbb M^*=\ilim{i\in I}  \mathcal N_i^*$.

\end{theorem}

\begin{proof} $\Rightarrow)$  $I$ is a set because it is a subset of the set of quotient $K$-modules of $\mathbb M(K)$, by \ref{W1}. Given two quotient $\mathcal K$-modules $\mathbb M\to\mathcal N_1,\mathcal N_2$,  the image, $\mathcal N_3$, of the obvious morphism $\mathbb M\to \mathcal N_1\times \mathcal N_2$ is a quotient   $\mathcal K$-module of  $\mathbb M$ and $\mathcal N_3\leq \mathcal N_1,\mathcal N_2$. Therefore, $I$ is a downward directed set.
Let $S$ be a commutative $K$-algebra,
the morphism $$\ilim{i\in I}  \mathcal N_i^*(S)\to \Hom_{\mathcal K}(\mathbb M,\mathcal S)\overset{\text{\ref{adj2}}}=\mathbb M^*(S)$$ is obviously injective, and it is surjective by  Theorem \ref{3.14}. Hence, $\mathbb M^*=\ilim{i\in I}  \mathcal N_i^*$.

$\Leftarrow)$ Observe that
$$\aligned \Hom_{\mathcal K}(\mathbb M,\mathcal N) & \overset{\text{\ref{trivial}}}=\Hom_{\mathcal K}(\mathcal N^*,\mathbb M^*)=\Hom_{\mathcal K}(\mathcal N^*,\ilim{i} \mathcal N_i^*)\overset{\text{\ref{L5.11}}}=\ilim{i}\Hom_{\mathcal K}(\mathcal N^*,\mathcal N_i^*)\\ & \overset{\text{\ref{trivial}}} =\ilim{i}\Hom_{\mathcal K}(\mathcal N_i,\mathcal N).\endaligned$$
Then, every morphism $\mathbb M\to \mathcal N$ factors through some $\mathcal N_i$ and then its cokernel is a quasi-coherent module. By Lemma \ref{lemaD}, $\mathbb M$ is  dually separated.

\end{proof}

\begin{corollary} \label{QHom2}Let  $R=K$ be a field.
 If $\mathbb M$ is  dually separated, then $\mathbb M^*$ is  dually separated.\end{corollary}

\begin{proof} It is a consequence of Theorem \ref{P2}, Example \ref{3.3.3} and Proposition \ref{3.3.4}.\end{proof}

\section{Reflexive $\mathcal R$-modules}

\begin{proposition} \label{dualq3} Let $\mathbb M$ be a reflexive $\mathcal R$-module. $\mathbb M$ is  dually separated    
iff there exist a subset $I$ and  a monomorphism $\mathbb M^*\hookrightarrow \prod^I\mathcal R$.
\end{proposition}

\begin{proof} Let $\mathbb M$ be  dually separated.
Consider an epimorphism
$\oplus^IR\to \mathbb M(R)$. The composite morphism $\mathbb M^*\hookrightarrow {\mathbb M_{qc}}^*\hookrightarrow \prod^I\mathcal R$ is  a monomorphism.

Now, let  $\mathbb M^*\hookrightarrow  \prod^I\mathcal R$ be   a monomorphism. The dual morphism $ \oplus^I\mathcal R\to \mathbb M$, factors as follows: $ \oplus^I\mathcal R\to {\mathbb M_{qc}}\to \mathbb M$. Dually, we have $\mathbb M^*\to {\mathbb  M_{qc}}^*\to  \prod^I\mathcal R$. Therefore, the morphism $\mathbb M^*\to {\mathbb  M_{qc}}^*$ is  a monomorphism and  $\mathbb M$  is  dually separated.

\end{proof}

\begin{definition} An $\mathcal R$-module $\mathbb M$ is said to be (linearly) separated if
for each commutative $R$-algebra  $S$ and  $m\in\mathbb M(S)$ there exist a commutative $S$-algebra $T$ and a $w\colon \mathbb M\to \mathcal T$ such that $w(m)\neq 0$ (that is,
the natural morphism
$\mathbb M\to \mathbb M^{**}$, $m\mapsto \tilde m$, where $\tilde m(w):=w(m)$ for any $w\in\mathbb M^*$, is  a monomorphism).\end{definition}

Every $\mathcal R$-submodule of a separated $\mathcal R$-module is separated.

\begin{example} If $\mathbb M$ is a dual  $\mathcal R$-module, then it is separated: Given $0\neq w\in \mathbb M=\mathbb N^*$, there exists an $n\in \mathbb N$ such that $w(n)\neq 0$.
Let $\tilde n\in \mathbb M^*$ be defined by $\tilde n(w'):=w'(n)$, for any $w'\in \mathbb M$. Then, $\tilde n(w)\neq 0$.\end{example}

\begin{proposition}\label{2125} Let $R=K$ be a field and let $\mathbb M$ be a $\mathcal K$-module
such that $\mathbb M^*$ is well defined.
$\mathbb M$ is separated iff the natural 
morphism $\mathbb M \to \mathbb M_{sch}$ is  a monomorphism.
Therefore, $\mathbb M$ is separated iff it is a $\mathcal K$-submodule of a $\mathcal K$-module scheme.
\end{proposition}

\begin{proof}
Assume $\mathbb M$ is separated. Let $m \in \mathbb M(S)$ be such that
$m = 0$ in ${\mathbb M_{sch}}(S)$. $ {\mathbb M_{sch}}(S)\overset{\text{\ref{1211b}}}= {\rm Hom}_K
(\mathbb M^*(K), S)$, then $m(w):= w(m)= 0$ for any $w \in \mathbb M^*(K)$.

Let $T$ be a commutative $S$-algebra, and let $\{ e_i\}_{i\in I}$ be a $K$-basis of $T$. Consider the composite morphism
$$\mathbb M^*(T)
\overset{\text{\ref{adj2}}}= {\rm Hom}_{\mathcal K} (\mathbb M, {\mathcal T}) = {\rm Hom}_{\mathcal K} (\mathbb M,
\oplus_I{{\mathcal K}}) \subset \prod_I {\rm Hom}_{\mathcal K} (\mathbb M, {{\mathcal K}}),$$ which
assigns to every $w_T \in \mathbb M^*(T)$ a $(w_i) \in \prod \mathbb M^*(K)$.
Specifically,  $w_T(m')= \sum_i w_i(m') \cdot
e_i$, for any $m' \in \mathbb M(T)$. Therefore, $w_T(m)=0$ for any $w_T \in \mathbb M^*(T)$. As $\mathbb M$ is separated, this means that
$m = 0$, i.e., the morphism $\mathbb M \to {\mathbb M_{sch}}$ is  a monomorphism.

Now, assume $\mathbb M \to {\mathbb M_{sch}}$ is  a monomorphism. Observe that ${\mathbb M}_{sch}$ is separated because it is reflexive. Then, $\mathbb M$ is separated.

Finally, the second statement of the proposition is obvious.
\end{proof}

\begin{theorem} \label{RPQ} Let $R=K$ be a field. $\mathbb M$ is a reflexive $\mathcal K$-module iff
$\mathbb M$ is equal to the inverse limit of its quasi-coherent quotient $\mathcal R$-modules.
\end{theorem}

\begin{proof} Suppose that $\mathbb M$ is reflexive. $\mathbb M^*$ is separated, because it is a dual $\mathcal R$-module.  By Proposition \ref{2125}, the morphism
$\mathbb M^*\to {\mathbb M^*}_{sch}={\mathbb M_{qc}}^* $ is  a monomorphism.
Then, $\mathbb M$ is  dually separated.
Let $\{\mathcal M_i\}_{i\in I}$ be the set of  all
quasi-coherent quotient modules of $\mathbb M$. Then,
$\mathbb M^*=\ilim{i\in I}  \mathcal M_i^*$, by Theorem \ref{P2}. Therefore,
$$\mathbb M=\mathbb M^{**}=\plim{i\in I}  \mathcal M_i.$$

Suppose now that $\mathbb M$ is equal to the inverse limit of its quasi-coherent quotient $\mathcal K$-modules, $\mathbb M=\plim{i} \mathcal N_i$. Then, $\mathbb M=(\ilim{i} \mathcal N_i^*)^*$ is  dually separated, by \ref{3.3.4} and \ref{QHom2}.
By Theorem \ref{P2}, $\mathbb M^*=\ilim{i} \mathcal N_i^*$ and $\mathbb M=\plim{i} \mathcal N_i=\mathbb M^{**}$.
\end{proof}

Let $R={\mathbb Z}$ and $M={\mathbb Z}/2{\mathbb Z}$. Then, $\mathbb M:=\mathcal M^*$ is reflexive but it is not  dually separated, because ${\mathbb M_{qc}}^*=0$,  because $\mathbb M(R)=0$.

\section{Proquasi-coherent modules}

\begin{definition} An $\mathcal R$-module is said to be a proquasi-coherent module if it is an inverse limit of quasi-coherent $\mathcal R$-modules.\end{definition}

In this section, $K$ will be a field.

\begin{example} Reflexive  $\mathcal K$-modules are proquasi-coherent, by Theorem \ref{RPQ}.

\end{example}

\begin{proposition} \label{P6} If $\mathbb M$ is a  proquasi-coherent $\mathcal K$-module, then it is a dual  $\mathcal K$-module and it is a direct  limit of $\mathcal K$-module shemes.
In particular, proquasi-coherent $\mathcal K$-modules are  dually separated.
\end{proposition}

\begin{proof} $\mathbb M=\plim{} \mathcal M_i=(\ilim{} \mathcal M_i^*)^*$. $\ilim{} \mathcal M_i^*$ is  dually separated    by Example \ref{3.3.3} and Proposition \ref{3.3.4}. Then, its dual, which is $\mathbb M$, is a direct  limit of $\mathcal K$-module shemes, by Theorem \ref{P2} and it is  dually separated   by Corollary \ref{QHom2}.\end{proof}

\begin{proposition} \label{K1} 
Let $\mathbb P$ be a proquasi-coherent  $\mathcal K$-module and $\mathbb M$ a separated $\mathcal K$-module. Let $f\colon \mathbb P\to \mathbb M$ be a morphism of $\mathcal K$-modules.  Then,
$\Ker f$ is proquasi-coherent.\end{proposition}

\begin{proof} By Theorem \ref{2125}, there exist  a $K$-vector space $V$  and  a monomorphism
$\mathbb M\hookrightarrow \mathcal V^*$. We can assume $\mathbb M=\mathcal V^*= \prod_{I}\mathcal K$. Given $I'\subset I$, let $f_{I'}$ be the composition of
$f$ with the obvious projection $ \prod_{I}\mathcal K\to  \prod_{I'}\mathcal K$.
Then,
$$\Ker f=\plim{I'\subset I,\, \#I'<\infty} \Ker f_{I'}$$
It is sufficient to prove that $\Ker f_{I'}$ is proquasi-coherent, since the inverse limit of proquasi-coherent modules is proquasi-coherent. Let us write $I'=I''\coprod \{i\}$. $\Ker f_{I'}$ is the kernel of the composite morphism $\Ker f_{\{i\}}\hookrightarrow \mathbb P\to \prod_{I''}\mathcal K$.
By induction on $\# I'$, it is sufficient to prove that $\Ker f_{i}$ is proquasi-coherent.
Let us write $f=f_{\{i\}}$.

If  $f\colon \mathbb P\to\mathcal K$ is the zero morphism the proposition  is obvious. Assume $f\neq 0$.
Then, $f$ is an epimorphism (because $\mathbb P$ is dually separated). Let us write $\mathbb P= \plim{i} \mathcal V_i$ and let $v=(v_i)\in \plim{i} V_i=\mathbb P(K)$ be a vector such that $f_K((v_i))\neq 0$. Then, $\mathbb P=\Ker f\oplus\mathcal K\cdot v$.
Let $\bar V_i:=V_i/\langle v_i\rangle$.
Let us prove that $\Ker f\simeq \plim{i}\bar{\mathcal V}_i$: Let $i'$ be such that $v_{i'}\neq 0$. Consider the exact sequences
$$0\to\mathcal K \cdot v_i\to \mathcal V_i\to \bar{\mathcal V}_i\to 0, \qquad (i>i')$$
Dually, we have the exact sequences
$$0\to \bar{\mathcal V}_i^*\to \mathcal V_i^* \to \mathcal K\to 0$$
Taking the direct limit we have the exact sequence
$$0\to \ilim{i} (\bar{\mathcal V}_i^*)\to \ilim{i} (\mathcal V_i^*) \to \mathcal K\to 0$$
Dually, we have the exact sequence
$$0\to \mathcal K\cdot v\to \mathbb P\to \plim{i} \bar{\mathcal V}_i\to 0$$
Then, $\Ker f\to \plim{i}\bar{\mathcal V}_i$, $(v_i)_i\mapsto (\bar v_i)_i$ is an isomorphism.

\end{proof}

\begin{proposition} \label{5.5} Every direct summand of a pro\-qua\-si-co\-herent module is proquasi-coherent.\end{proposition}

\begin{theorem} \label{P=DQ} Let $\mathbb M$ be a $\mathcal K$-module. $\mathbb M$ is  proquasi-coherent
iff $\mathbb M$ is a dual $\mathcal K$-module and it is  dually separated. \end{theorem}

\begin{proof} By Proposition \ref{P6}, we  only have  to prove the sufficiency.
 Let us write $\mathbb M=\mathbb N^*$. The dual morphism of the natural
 morphism $\mathbb N\to \mathbb N^{**}$ is a retraction of the natural morphism $\mathbb M\to \mathbb M^{**}$. Then, $\mathbb M^{**}=\mathbb M\oplus\mathbb M'$. By Proposition \ref{5.5}, $\mathbb M$ is proquasi-coherent, because $\mathbb M^{**}$ is proquasi-coherent by Theorem \ref{P2}.

\end{proof}

\begin{theorem} \label{P=DQ2} A  $\mathcal K$-module is proquasi-coherent iff it is the dual $\mathcal K$-module of  a dually separated $\mathcal K$-module.\end{theorem}

\begin{proof} If $\mathbb M=\plim{i}\mathcal M_i$ is proquasi-coherent, then $\mathbb M=(\ilim{i} \mathcal M_i^*)^*$.
$\ilim{i} \mathcal M_i^*$ is  dually separated    and $\mathbb M=(\ilim{i} \mathcal M_i^*)^*$.

If $\mathbb M'$ is  dually separated,
then $\mathbb M'^*$ is  dually separated, by Corollary \ref{QHom2}.
By Theorem \ref{P=DQ}, $\mathbb M'^*$ is proquasi-coherent.\end{proof}

\begin{proposition} \label{P4} If $\mathbb P,\mathbb P'$ are proquasi-coherent $\mathcal K$-modules, then  $\mathbb Hom_{\mathcal K}(\mathbb P, \mathbb P')$ is proquasi-coherent. In particular, $\mathbb P^*$ is proquasi-coherent.

\end{proposition}

\begin{proof} Let us write $\mathbb P=\ilim{i} \mathcal V_i^*$ and $\mathbb P'=\plim{j} \mathcal V'_j$.
Then,

$$\mathbb Hom_{\mathcal K}(\mathbb P, \mathbb P')=\mathbb Hom_{\mathcal K}(\ilim{i} \mathcal V_i^*, \plim{j} \mathcal V'_j)=\plim{i,j} \mathbb Hom_{\mathcal K}( \mathcal V_i^*,  \mathcal V'_j)=\plim{i,j} (\mathcal V_i\otimes\mathcal V'_j)$$
Hence, $\mathbb Hom(\mathbb P, \mathbb P')$ is proquasi-coherent.

\end{proof}

\begin{proposition} \label{P9} Let $\mathbb A$ be a $\mathcal K$-algebra  and  dually separated,
and let $\mathbb P,\mathbb P'$ be proquasi-coherent $\mathcal K$-modules and $\mathbb A$-modules.
Then, a morphism of $\mathcal K$-modules, $f\colon \mathbb P\to\mathbb P'$, is a morphism
of $\mathbb A$-modules iff $f_K\colon \mathbb P(K)\to\mathbb P'(K)$ is a morphism of $\mathbb A(K)$-modules.

\end{proposition}

\begin{proof} Proceed  as in the proof of Proposition \ref{P10} (2).

\end{proof}

\begin{lemma} \label{main} Let $M$ be an $R$-module. Then,
$$\mathbb Hom_{\mathcal R}(\prod_I\mathcal R,\mathcal M)=\oplus_I \mathbb Hom_{\mathcal R}(\mathcal R,\mathcal M)=\oplus_I\mathcal M$$

\end{lemma}

\begin{proof} $\mathbb Hom_{\mathcal R}(\prod_I\mathcal R,\mathcal M)=\mathbb Hom_{\mathcal R}((\oplus_I\mathcal R)^*,\mathcal M)\overset{\text{\ref{prop4}}}=(\oplus_I\mathcal R)\otimes \mathcal M=\oplus_I\mathcal M$.\end{proof}

\begin{lemma} \label{refpro} Let $\{\mathbb M_i\}_{i\in I}$ be a set of dual  $\mathcal R$-modules and let $N$ be an $R$-module. Then,
$$\Hom_{\mathcal R}(\prod_{i\in I} \mathbb  M_i,\mathcal N)=
\oplus_{i\in I} \Hom_{\mathcal R}(\mathbb  M_i,\mathcal N)$$
In particular, $(\prod_{i\in I}\mathbb  M_i)^*=\oplus_{i\in I}\mathbb  M_i^*$ and if
$\mathbb M_i$ is reflexive, for every $i$, then $\prod_{i\in I}\mathbb  M_i$ is reflexive.

\end{lemma}

\begin{proof} Let $f\in \Hom_{\mathcal R}(\prod_{i\in I} \mathbb  M_i,\mathcal N)$ and $f_i:=f_{|\mathbb  M_i}$. If
$f_{|\oplus_{i\in I}\mathbb  M_i}=0$, then $f=0$: 
Given
$m=(m_i)_{i\in I}\in \prod_{i\in I} \mathbb M_i(S)$, let
$g\colon \prod_{i \in I} \mathcal S\to \mathcal N_{|S}$, $g_T((t_i)_i):=f_T((t_i\cdot m_i)_i)$, for every commutative $S$-algebra $T$. Since $g_{|\oplus_i \mathcal S}=0$, then $g=0$, by Proposition \ref{main}. Therefore, $f=0$.

Consider the obvious inclusion morphism
$$
\oplus_{i\in I} \Hom_{\mathcal R}(\mathbb  M_i,\mathcal N)\subseteq
\Hom_{\mathcal R}(\prod_{i\in I} \mathbb  M_i,\mathcal N).$$

 Let $J:=\{i\in I\colon f_i:=f_{|\mathbb  M_i}\neq 0\}$. For each $j\in J$, let $R_j$ be a commutative $R$-algebra and $m_j\in\mathbb M_j(R_j)$  such that $0\neq f_j(m_j)\in N\otimes_RR_j$. Let $S:=\prod_{j\in J}R_j$. The obvious morphism of $R$-algebras $S\to R_i$ is surjective, and this morphism of $R$-modules has a section. Write $\mathbb M_i=\mathbb N_i^*$.  The natural morphism $$\pi_i\colon \mathbb M_i(S)=\Hom_{\mathcal R}(\mathbb N_i,\mathcal S)\to \Hom_{\mathcal R}(\mathbb N_i,\mathcal R_i)=
\mathbb M_i(R_i)$$ has a section of $R$-modules. Let $m'_i\in \mathbb M_i(S)$ be such that $\pi_i(m'_i)=m_i$.
The morphism of $\mathcal S$-modules $g\colon \prod_{J}\mathcal S\to \mathcal N_{|S}$, $g((s_j)):=f((s_j\cdot m'_j)_j)$ satisfies that
$g_{|\mathcal S}\neq 0 $, for every factor $\mathcal S\subset \prod_{J} \mathcal S$.
Then, by Proposition \ref{main},  $\# J<\infty$.

Finally, define $h:=\sum_{j\in J} f_j\in \oplus_{i\in I} \Hom_{\mathcal R}(\mathbb  M_i,\mathcal N)$, then $f=h$.

\end{proof}

\begin{proposition} \label{Fcansado} Let $I$ be a totally ordered set and  $\{f_{ij}\colon M_i\to M_j\}_{i\geq j\in I}$ an inverse system of $K$-vector spaces. Then, $\plim{i} \mathcal M_i$ is reflexive.
\end{proposition}

\begin{proof}
$\plim{i} \mathcal M_i$ is a direct limit of submodule schemes  $\mathcal V_j^*$, by \ref{P2} and \ref{P=DQ2}.
If all the vector spaces $V_j$ are finite dimensional then
$\plim{i} \mathcal M_i$ is quasi-coherent, then it is reflexive. In other case,
there exists an injective morphism  $f\colon \prod_{\mathbb N}\mathcal K\hookrightarrow \plim{i} \mathcal M_i$. Let $\pi_j\colon \plim{i} \mathcal M_i\to \mathcal M_j$ be the natural morphisms. Let $g_r\colon \mathcal K^r\hookrightarrow \prod_{\mathbb N}\mathcal K$ be
defined by $g_r(\lambda_1,\cdots,\lambda_r):=(\lambda_1,\cdots,\lambda_r,0,\cdots,0,\cdots)$.
Let  $i_1\in I$ be such that $\pi_{i_1}\circ f\circ g_1$ is a monomorphism. Recursively,
let $i_n>i_{n-1}$ be such that $\pi_{i_n}\circ f\circ g_n$ is  a monomorphism.
If there exists a $j>i_n$ for any $n$, the composite morphism $\oplus_{\mathbb N}\mathcal K\subset\prod_{\mathbb N}\mathcal K \to \mathcal M_j$ is  a monomorphism,
and by Proposition \ref{main} the morphism  $\prod_{\mathbb N}\mathcal K \to \mathcal M_j$ factors through the projection onto some $\mathcal K^r$, which is contradictory. Therefore, $\plim{i} \mathcal M_i=\plim{n\in\mathbb N} \mathcal M_{i_n}$.

 Let $\mathcal M'_{i_r}$ be the image of  $\plim{n} \mathcal M_{i_n}$ in $\mathcal M_{i_r}$. Then, $\plim{n} \mathcal M'_{i_n}=\plim{n} \mathcal M_{i_n}$.
 Let $H_n:=\Ker [M'_{i_n}\to M'_{i_{n-1}}]$. Then, $\plim{n} \mathcal M_{i_n}\simeq\prod_n \mathcal H_n$. By Lemma \ref{refpro}, $\plim{n} \mathcal M_{i_n}$ is reflexive.

\end{proof}

\section{Flat SML $R$-modules and  dually separated   $\mathcal R$-modules}

\begin{theorem} \label{5.1} $\mathcal M^*$ is  dually separated    iff 
the morphism
$$M\otimes_RN\to \Hom_R(M^*,N)$$
is injective, for any $R$-module $N$. \end{theorem}

\begin{proof} The morphism  $M\otimes_RN\overset{\text{\ref{prop4}}}=\Hom_{\mathcal R}(\mathcal M^*,\mathcal N)\to \Hom_R(M^*,N)$ is injective, for any $R$-module $N$  iff $\mathcal M^*$ is  dually separated, by Theorem \ref{W1}.
\end{proof}

\begin{corollary} \label{SQF}
If $\mathcal M^*$ is  dually separated, then $M$ is a flat $R$-module and the morphism $M\to M^{**}$ is universally injective, that is, $M\otimes_R S\to M^{**}\otimes_R S$ is injective for every commutative $R$-algebra $S$.
\end{corollary} 

\begin{proof} $M\otimes -$ is a left exact functor because $\Hom_R(M^*,-)$ is a left exact functor. Hence, $M$ is flat. Finally, the composite morphism, $$M\otimes_R S\to M^{**}\otimes_R S\to  \Hom_R(M^*,S)$$ is injective, then $M\otimes_R S\to M^{**}\otimes_R S$ is injective.

\end{proof}

Let $R={\mathbb Z}$ and $M={\mathbb Q}$, which is a flat $\mathbb Z$-module. $\mathcal M^*$ is not    dually separated, because $M\to M^{**}$ is the zero morphism, because $M^*=0$.

\begin{corollary} Let  $\mathcal M^*$   be  dually separated. Then, the morphism $$M\otimes N^*\to \Hom_R(N,M)$$
is injective, for any $R$-module $N$. 

\end{corollary} 

\begin{proof} The composite morphism
$$M\otimes N^*\to \Hom_R(N,M)\to \Hom_R(M^*,N^{*})$$
is injective, then $M\otimes N^*\to \Hom_R(N,M)$ is injective.
\end{proof}

\begin{theorem} $\mathcal M^*$   is   dually separated     iff  the natural morphism
$$M\otimes_RS\to (M\otimes_RS)^{**}:=\Hom_S(\Hom_S(M\otimes_R S,S),S)$$
is injective, for any commutative $R$-algebra $S$.
\end{theorem}

\begin{proof} It is an immediate consequence of Theorem \ref{4.3B}.
%

\end{proof}

\begin{proposition} \cite[Prop. 5.3]{Garfinkel} \label{4.4} $\mathcal M^*$   is   dually separated     iff there exists  a monomorphism  $\mathcal M\hookrightarrow \prod^I\mathcal R$.\end{proposition}

\begin{proof}  It is an immediate consequence of Proposition \ref{dualq3}.\end{proof}

\begin{example} \label{Ejem} Let $P$ be a projective module, then $\mathcal P^*$ is  dually separated: $P$ is a direct summand of a free module $\oplus^I R$.
Then, $\mathcal P\subseteq \oplus^I\mathcal R\subseteq \prod^I\mathcal R$ and
$\mathcal P^*$ is  dually separated.

\end{example}

\begin{corollary} Let  $N\hookrightarrow M$ be a universally injective morphism of $R$-modules. If  $\mathcal M^*$ is   dually separated,  $\mathcal N^*$ is   dually separated.

\end{corollary}

\begin{proof} $N\hookrightarrow M$ is a universally injective morphism of $R$-modules iff $\mathcal N\to \mathcal M$ is a monomorphism. The collorary is an inmediate consequence of Proposition \ref{4.4}.

\end{proof}

Noetherian rings are coherent rings (see \cite[I 6-7]{Lubkin}) for definition and properties).

\begin{theorem} Let $R$ be a coherent ring and $M$ an $R$-module. $\mathcal M^*$ is   dually separated     iff  there exists an inclusion $M\subseteq \prod^I R$ such that the cokernel is flat.\end{theorem}

\begin{proof}  Observe that $\Hom_R(M,\prod^I R)=\Hom_{\mathcal R}(\mathcal M,\prod^I \mathcal R)$.
$\prod^IR$ is a flat $R$-module and for every $R$-module $S$ the natural morphism $(\prod^I R)\otimes_R S\to  \prod^I S$ is injective, because $R$ is a coherent ring. Then, a morphism $M\to \prod^I R$ is injective and the cokernel is a flat module iff $\mathcal M\to \prod^I\mathcal R$ is a monomorphism.

Then, this theorem is a immediate consequence of Proposition \ref{4.4}.

\end{proof} 

\begin{lemma} \label{ImCk} Let $f\colon \mathcal M^*\to\mathcal N$ be a morphism of $\mathcal R$-modules.
Then, $\Coker f$ is quasi-coherent iff $f$ factors through the quasi-coherent module associated with $\Ima f_R$.\end{lemma}

\begin{proof}
Let $N_1=\Ima f_R$ and let $N_2=N/N_1$. Observe that  $\Coker f$ is quasi-coherent iff 
 $\Coker f=\mathcal N_2$, and   $\Coker f=\mathcal N_2$ iff the composite morphism $\mathcal M^*\to \mathcal N\to\mathcal N_2$ is zero. Consider the diagram
 $$\xymatrix{  \Hom_{\mathcal R}(\mathcal M^*,\mathcal N_1) \ar[r] &
\Hom_{\mathcal R}(\mathcal M^*,\mathcal N) \ar[r] & \Hom_{\mathcal R}(\mathcal M^*,\mathcal N_2) &  \\  M\otimes_R N_1 \ar[r]  \ar@{=}[u]^-{\text{\ref{prop4}}}&
M\otimes_R N \ar[r] \ar@{=}[u]^-{\text{\ref{prop4}}}&  M\otimes_R N_2 \ar@{=}[u]^-{\text{\ref{prop4}}}
\ar[r] & 0}$$
Then,  the composite morphism $\mathcal M^*\to \mathcal N\to\mathcal N_2$ is zero iff 
$f$ factors through $\mathcal N_1$, which is the quasi-coherent module associated with $\Ima f_R$.
We are done.
\end{proof}

\begin{remarks} If $f\colon \mathcal M^*\to \mathcal N$ is an epimorphism, $N$ is a finitely generated module: 
$f=\sum_{i=1}^r m_i\otimes n_i\in \Hom_{\mathcal R}(\mathcal M^*,\mathcal N)= M\otimes N$, therefore
$f$ factors through the coherent module associated with $\langle n_1,\ldots,n_r\rangle$, then $N=\langle n_1,\ldots,n_r\rangle$.

If $N_1\hookrightarrow N_2$ is an injective morphism of $R$-modules and $M$ is flat, the map
$\Hom_{\mathcal R}(\mathcal M^*,\mathcal N_1)=M\otimes_R N_1\to M\otimes_R N_2= \Hom_{\mathcal R}(\mathcal M^*,\mathcal N_1)$ is injective.

\end{remarks}

\begin{theorem} \label{QQ} $\mathcal M^*$ is  dually separated     iff every morphism $f\colon \mathcal M^*\to \mathcal N$ (uniquely) factors through the coherent module associated with $\Ima f_R$.\end{theorem}

\begin{proof} It is an immediate consequence of \ref{lemaD} and \ref{ImCk}.

%
%

\end{proof}

\begin{theorem} \label{trace1} $\mathcal M^*$ is dually separated iff any morphism $f\colon \mathcal M^*\to\mathcal R$
factors through the quasi-coherent module associated with  $\Ima f_R$.
\end{theorem}

\begin{proof}  $\Rightarrow)$ It is an immediate consequence of \ref{QQ}.

$\Leftarrow)$ We have to prove that a morphism $f\colon \mathcal M^*\to \mathcal N$ is zero if
$f_R=0$, by \ref{W1}. Any morphism $f\colon \mathcal M^*\to \mathcal N$ factors through the quasi-coherent module associated with a finitely generated submodule of $N$. Then, we can suppose that $N$ is finitely generated, that is, $N=\langle n_1,\ldots,n_r\rangle$. 

Let us proceed by induction on $r$. If $r=1$, $N\simeq R/I$, for some ideal  $I\subset R$. 
Let $\pi\colon \mathcal R\to \mathcal N$ be the quotient morphism. There exists a morphism
$g\colon \mathcal M^*\to \mathcal R$ such that the diagram
$$\xymatrix{\mathcal M^* \ar[r]^-g \ar[rd]_-f & \mathcal R \ar[d]^-\pi \\ &  \mathcal N}$$
is commutative (recall $\Hom_{\mathcal R}(\mathcal M^*,\mathcal N')\overset{\text{\ref{prop4}}}=M\otimes_R N'$).
Then, $\Ima g_R\subseteq I$, because $\Ima(\pi_R\circ g_R)=\Ima (\pi\circ g)_R=\Ima f_R=0$. 
Then, $g$ factors through $\mathcal I$ and $f=0$.
Assume the statement is true for $1,\ldots,r-1$ and $N=\langle n_1,\ldots,n_r\rangle$. 
Let $N'=N/\langle n_1\rangle$ and let $\pi\colon \mathcal N\to \mathcal N'$ be the quotient morphism.
Observe that $(\pi\circ f)_R=\pi_R\circ f_R=0$, then  $\pi\circ f=0$, by the induction hypothesis. 
Let $\mathcal N_1$ be the quasi-coherent module associated with $\langle n_1\rangle$. Consider the diagram
$$\xymatrix{M\otimes_R \langle n_1\rangle\ar[r] & M\otimes_R N\ar[r]&  M\otimes_R N' \ar[r] & 0\\
\Hom_{\mathcal R}(\mathcal M^*,\mathcal N_1) \ar@{=}[u]^-{\text{\ref{prop4}}} \ar[r] & 
\Hom_{\mathcal R}(\mathcal M^*,\mathcal N) \ar@{=}[u]^-{\text{\ref{prop4}}} \ar[r]^-{\pi_*}&  \Hom_{\mathcal R}(\mathcal M^*,\mathcal N') \ar@{=}[u]^-{\text{\ref{prop4}}} & }$$
Since $\pi_*(f)=\pi\circ f=0$, $f$ factors through a morphism $g\colon \mathcal M^*\to \mathcal N_1$.
Observe that $g_R=0$, because $f_R=0$,  then $g=0$ and $f=0$.

\end{proof} 

A module $M$ is a trace module if every $m\in M$
holds $m\in M^*(m)\cdot M$, where $M^*(m):=\{w(m)\in R\colon w\in M^*\}$ (see \cite{Garfinkel}).   

\begin{proposition} \label{trace2} $M$ is a trace module iff any morphism $f\colon \mathcal M^*\to\mathcal R$
factors through the quasi-coherent module associated with  $\Ima f_R$.
\end{proposition}

\begin{proof} $\Hom_{\mathcal R}(\mathcal M^*,\mathcal R)=M$, then $f=m\in M$ and $\Ima f_R=M^*(m)$.
Let $I\subseteq R$ be an ideal, then $f=m$ factors through $\mathcal I$ iff $m\in I\cdot M$, as it is easy to see taking into account the following diagram
$$\xymatrix{\Hom_{\mathcal R}(\mathcal M^*,\mathcal I) \ar[r] \ar@{=}[d] &
\Hom_{\mathcal R}(\mathcal M^*,\mathcal R) \ar@{=}[d] \\ I\otimes_R M \ar[r] & M}$$
We are done.

\end{proof}

\begin{corollary}  $\mathcal M^*$ is dually separated iff $M$ is a trace module.

\end{corollary}

\begin{proof} It is an immediate consequence of \ref{trace1} and \ref{trace2}.

\end{proof}

\begin{lemma} \label{lemaP} Let $M$ be a flat $R$-module and $P$ a finitely presented $R$-module. Then,
$$\Hom_{\mathcal R}(\mathcal M^*,{\mathcal P^*}_{qc})=\Hom_{\mathcal R}(\mathcal M^*,\mathcal P^*).$$

\end{lemma}

\begin{proof} Consider an exact sequence of morphisms $R^n\to R^m\to P\to 0$. Dually,
$0\to \mathcal P^*\to \mathcal R^m\to \mathcal R^n$ is exact. From the commutative diagram of exact rows
$$\xymatrix{0 \ar[r] & \Hom_{\mathcal R}(\mathcal M^*,\mathcal P^*)\ar[r] &  \Hom_{\mathcal R}(\mathcal M^*,\mathcal R^m) \ar[r]&  \Hom_{\mathcal R}(\mathcal M^*,\mathcal R^n)\\ 0 \ar[r] & P^* \otimes_R M \ar[r] & R^m\otimes_R M \ar[r] \ar@{=}[u]^-{\text{\ref{prop4}}} & R^n\otimes_R M  \ar@{=}[u]^-{\text{\ref{prop4}}}\\ & \Hom_{\mathcal R}(\mathcal M^*, {\mathcal P^*}_{qc}) \ar@{=}[u]^-{\text{\ref{prop4}}} & & }$$
one has that $\Hom_{\mathcal R}(\mathcal M^*,{\mathcal P^*}_{qc})=\Hom_{\mathcal R}(\mathcal M^*,\mathcal P^*).$

\end{proof}

\begin{proposition} $\mathcal M^*$ is dually separated iff $M$ is a flat strict Mittag-Leffler module.

\end{proposition}

\begin{proof} Let $\{P_i\}$ be a direct system of finitely presented modules such that $M=\ilim{i} P_i$. Then, $\mathcal M^*=\plim{i} \mathcal P_i^*$.  Observe that
 $$\Hom_{\mathcal R}(\mathcal M^*,\mathcal N)\overset{\text{\ref{prop4}}}=M\otimes_R N=\ilim{i} P_i\otimes_R N \overset{\text{\ref{prop4}}}=\ilim{i}
 \Hom_{\mathcal R}(\mathcal P_i^*,\mathcal N).$$

$\Rightarrow)$  $M$ is flat, by \ref{SQF}. The natural morphism $\mathcal M^*\to \mathcal P_i^*$ factors through $\mathcal M^*\to {\mathcal P_i^*}_{qc}$, by \ref{lemaP}. The morphism $\mathcal M^*\to 
 {\mathcal P_i^*}_{qc}$ factors through an epimorphism $\mathcal M^*\to \mathcal N$, by \ref{QQ}.
 $\mathcal M^*\to \mathcal N$ factors through the natural morphism $\mathcal M^*\to \mathcal P_j^*$, for some $j$.
We have the morphisms
 $$\mathcal M^*\to \mathcal P_j^*\to \mathcal N\to \mathcal P_i^*$$
 (recall $\mathcal M^*\to \mathcal N$ is an epimorphism). 
 Then,
 $\Ima({\mathcal M^*}(S)\to {\mathcal P_i^*}(S))=\Ima({\mathcal P_j^*}(S)\to {\mathcal P_i^*}(S))$, for any  commutative $R$-algebra $S$. Taking $S=R\oplus Q$ (for any $R$-module $Q$), we obtain
 $$
\Ima(\Hom_R( M,Q)\to \Hom_R(P_i,Q))=\Ima(\Hom_R(P_j,Q)\to \Hom_R(P_i,Q)).$$ Hence, $M$ is a flat strict Mittag-Leffler module.

$\Leftarrow)$ Let $\{P_i\}$ be a direct system of finitely presented modules so that $M=\ilim{i} P_i$ and  for every $i$  there exists a $j\geq i$ such that
 $$
\Ima(M^*\to P_i^*)=\Ima(P_j^*\to P_i^*).$$ 
Let $\mathcal M^*\to \mathcal N$ be a morphism of $\mathcal R$-modules. $\mathcal M^*\to \mathcal N$ factors through the natural morphism $\mathcal M^*\to \mathcal \mathcal P_i^*$, for some $i$. There exists $j\geq i$ such that $
\Ima(M^*\to P_i^*)=\Ima(P_j^*\to P_i^*).$ 
Then, $$\Ima(M^*\to N)=\Ima(P_j^*\to N)=:N_j.$$ 
The natural morphism
$\mathcal M^*\to\mathcal P_j^*$ factors through a morphism $\mathcal  M^*\to{\mathcal P_j^*}_{qc}$, by \ref{lemaP}. We have the morphisms 
 $$\mathcal M^*\to {\mathcal P_j^*}_{qc}\to \mathcal P_j^*\to \mathcal N$$ 
The composite morphism $\mathcal N_j\to \Ima(\mathcal M^*\to\mathcal N)\subseteq \Ima  ({\mathcal P_j^*}_{qc}\to \mathcal N)$ is an epimorphism. Hence,  $\Ima(\mathcal M^*\to\mathcal N)=\Ima  ({\mathcal P_j^*}_{qc}\to \mathcal N)$. Therefore,  $\Coker(\mathcal M^*\to\mathcal N)=\Coker  ({\mathcal P_j^*}_{qc}\to \mathcal N)$, which is quasi-coherent. $\mathcal M^*$ is dually separated by \ref{lemaD}.

\end{proof}

It is well known that a module is a flat strict Mittag-Leffler module iff it is a trace module (see \cite[II. 2.3.4]{RaynaudG} and \cite[Th.3.2]{Garfinkel}).

\begin{proposition} \cite[Cor. 3]{Garfinkel} Let $M$ be a finitely generated module. Then, $\mathcal M^*$  is  dually separated    iff $M$ is a projective module.\end{proposition}

\begin{proof}  
$\Rightarrow)$ Let $\mathcal R^n\to \mathcal M$ be an epimorphism. The dual morphism $\mathcal M^*\to \mathcal R^n$ is a monomorphism and it factors through an epimorphism $\mathcal M^*\to \mathcal N$.
Then, $\mathcal M^*\simeq \mathcal N$ and by \cite{Amel2} $M$ is a projective module.

$\Leftarrow)$  See Example \ref{Ejem}.

\end{proof}

\begin{theorem}  \label{BB} Let $\mathcal M^*$ be   dually separated    and $\{\mathcal N_i\}$ the set of the coherent quotient $\mathcal R$-modules of $\mathcal M^*$. Then,
$\mathcal M =\ilim{i\in I}  \mathcal N_i^*.$
\end{theorem}

\begin{proof}   Proceed  as in the proof of Theorem \ref{P2}
to prove that $\mathcal M=\ilim{i\in I}  \mathcal N_i^*$. \end{proof} 

\begin{theorem} \label{5.11} $\mathcal M^*$ is   dually separated    iff $M$ is a flat Mittag-Leffler module and the morphism
$$M\otimes_R R/\mathfrak m\to \Hom_R(M^*,R/\mathfrak m)$$
is injective,  for every maximal ideal $\mathfrak m\subset R$.
\end{theorem} 
\begin{proof} $\Rightarrow)$ By Theorem  \ref{BB} and \cite[4.5]{Pedro2}, $M$ is a flat Mittag-Leffler module. Now, the direct part of this proposition is a consequence of Theorem \ref{5.1}.

$\Leftarrow)$ Let $f\colon \mathcal M^*\to \mathcal N$ be a morphism of $\mathcal R$-modules. By \cite[4.5,4.1]{Pedro2}, there exists a finitely generated submodule $N'\subset N$ such that $f$ factors through a morphism $f'\colon \mathcal M^*\to \mathcal N'$ and the dual morphism ${f'}^*\colon  {\mathcal N'}^*\to 
\mathcal M$ is a monomorphism. If we prove that $f'_R\colon M^*\to N'$ is an epimorphism, we are done by \ref{QQ}.
Assume $f'_R$ is not an epimorphism. By Nakayama's Lemma, there exists a maximal ideal $\mathfrak m\subset R$ such that the composite morphism
$M^*\to N'\to N'/\mathfrak m N'$ is not an epimorphism. Then  there exists an epimorphism  $N'/\mathfrak m N'\to R/\mathfrak m$
such that the composite morphism $M^*\to R/\mathfrak m$ is zero. 
Let $\tilde{R/\mathfrak m}$ be the quasi-coherent module associated with $R/\mathfrak m$. 
We have a morphism $\mathcal M^*\to \tilde{R/\mathfrak m}$
which is not zero (because the dual morphism is a monomorphism) 
and $M^*\to R/\mathfrak m$ is zero. This is contradictory because the composite morphism
$$\Hom_{\mathcal R}(\mathcal M^*,\tilde{R/\mathfrak m})=
M\otimes_R R/\mathfrak m\to \Hom_R(M^*,R/\mathfrak m)$$
is injective, by Theorem \ref{5.1}.
\end{proof}

%
%
%
%

\begin{theorem} \label{dl+}  Let $R$ be a noetherian ring. Let $M$ be a flat $R$-module such that
there exists a set of finitely generated submodules of $M$, $\{M_i\}$, so that
$M=\cup_{i\in I} M_i$ and the morphisms $M^*\to M_i^*$ are surjective. Then, $\mathcal M^*$ 
is   dually separated.
\end{theorem}

\begin{proof} Consider a morphism $f\colon \mathcal M^*\to \mathcal N$. Then,
$f=\sum_i m_i\otimes n_i\in M\otimes N=\Hom_{\mathcal R}(\mathcal M^*,\mathcal N)$.
Let $M_j$ be such that $m_i\in M_j$, for any $i$. Then, $f$ factors through $\mathcal M^*\to \mathcal M^*_j$. By \ref{lemaP}, $\Hom_{\mathcal R}(\mathcal M^*,{\mathcal M^*_j}_{qc})=\Hom_{\mathcal R}(\mathcal M^*,\mathcal M^*_j).$
Then, $f$  (uniquely) factors through a morphism
$\mathcal M^*\to {\mathcal M_j^*}_{qc}$. By the hypothesis, this morphism is an epimorphism. By Lemma \ref{QQ}, $\mathcal M^*$ is  dually separated.

\end{proof}

\begin{corollary} Let $R$ be a Dedekind domain. An $\mathcal R$-module $\mathcal M^*$  is  dually separated    iff  $M$ is the direct limit of its finitely generated projective submodules that are direct summands.\end{corollary}

\begin{proof} $\Rightarrow)$ Let $\pi\colon \mathcal M^*\to \mathcal N$ be an epimorphism.
Let $L= R^n\to N$ be an epimorphism and $g\colon \mathcal L\to \mathcal N$ the induced morphism. There exists a morphism $f\colon \mathcal M^*\to \mathcal L$ such that the diagram
$$\xymatrix{\mathcal M^*\ar[r]^-\pi \ar[rd]_-f & \mathcal N\\ & \mathcal L\ar[u]_-g}$$ is commutative, because the morphism
$\Hom_{\mathcal R}(\mathcal M^*,\mathcal L)=M\otimes_R L\to M\otimes_R N=\Hom_{\mathcal R}(\mathcal M^*,\mathcal N)$ is surjective. Let $L'=\Ima f_{\mathbb Z}\subseteq L$. Then, $L'$ is a finitely generated  projective module, the obvious morphism $\mathcal M^*\to \mathcal L'$ is an epimorphism and we have the commutative diagram
$$\xymatrix{\mathcal M^*\ar[r] \ar[rd]  & \mathcal N\\ & \mathcal L'\ar[u]}$$ 
Then, $\mathcal M^*$ is the inverse limit of its 
coherent quotient $\mathcal R$-modules $\mathcal L'$, such that $L'$ are
finitely generated  projective modules. Equivalently, $M$ is the direct limit of its finitely generated  projective submodules that are direct summands.

$\Leftarrow)$ It is a consequence of Theorem \ref{dl+}.

\end{proof}

\begin{corollary} Let $R$ be a local ring. $\mathcal M^*$  is  dually separated     iff $M$ is the direct limit of its finite free  submodules that are direct summands.\end{corollary}

\begin{proof}  $\Rightarrow)$ $\mathcal M^*$ is the inverse limit of its coherent quotient $\mathcal R$-modules. We only have to prove that every epimorphism $f\colon \mathcal M^*\to\mathcal N$ onto a coherent module factors through an epimorphism onto a free coherent module. Let ${\mathfrak m}$ be the maximal ideal of $R$. Let $R^n\to N$ be an epimorphism such that $R^n\otimes_RR/{\mathfrak m}\to N\otimes_RR/{\mathfrak m}$ is an isomorphism.
Let $\pi\colon \mathcal R^n\to\mathcal N$ be the induced epimorphism.
There exists a morphism $g\colon \mathcal M^*\to \mathcal R^n$ such that $\pi\circ g=f$, because the map
$$\Hom_{\mathcal R}(\mathcal M^*,\mathcal R^n)=M\otimes_R R^n\to M\otimes_R N=
\Hom_{\mathcal R}(\mathcal M^*,\mathcal N)$$
is surjective.  As $f_R\colon M^*\to N$ is an epimorphism, then $$(\Ima g_R)\otimes_R R/{\mathfrak m}\to R^n\otimes_R R/{\mathfrak m}=N\otimes_R R/{\mathfrak m}$$ is an epimorphism. By Nakayama's lemma $\Ima g_R=R^n$. Then, $f$ factors through the epimorphism $g\colon \mathcal M^*\to\mathcal R^n$.

$\Leftarrow)$ $M=\ilim{i} L_i$,  where $\{L_i\}$ is the set of finite free modules that are direct summands. Then, 
$\mathcal M^*=\plim{i} \mathcal L_i^*$. Let $f\colon \mathcal M^*\to\mathcal N$ be a morphism. Then,
$f\in M\otimes N=(\ilim{i}L_i)\otimes N=\ilim{i}(L_i\otimes N)$ and $f$ factors through an epimorphism $g\colon  \mathcal M^*\to\mathcal L^*_i$, for some $i$. Let $\pi\colon \mathcal L_i^*\to \mathcal N$ be a morphism such that $f=\pi\circ g$. $\Coker f=\Coker \pi$ is a quasi-coherent module. Then, $\mathcal M^*$ is    dually separated, by Theorem \ref{QQ}.

\end{proof}

%
%

\end{document}